\documentclass[graybox, envcountsame]{svmult}

\usepackage{mathptmx}       
\usepackage{helvet}         
\usepackage{courier}        
\usepackage{type1cm}        
\usepackage{makeidx}         
\usepackage{graphicx}        
\usepackage{multicol}        
\usepackage[bottom]{footmisc}

\usepackage{amsmath, amssymb}
\usepackage[color=green!40]{todonotes}
\usepackage{import}                       

\makeindex             
\smartqed              
\allowdisplaybreaks

\def\al{\alpha} 
 
\def\ga{\gamma} 
\def\de{\delta} 
\def\ep{\varepsilon} 
\def\ze{\zeta}

\def\si{\sigma} 
 
\def\ph{\varphi} 
 
\def\ps{\psi}

\def\De{\Delta}

\def\Om{\Omega}

\def\x{\times}
\def\p{\partial}

\def\H{\mathcal{H}}
\def\R{{\mathbb R}}

\let\on=\operatorname
\let\wt=\widetilde

\let\mc=\mathcal
\let\mf=\mathfrak
\newcommand{\ud}{\,\mathrm{d}}

\newcommand{\executeiffilenewer}[3]{%
\ifnum\pdfstrcmp{\pdffilemoddate{#1}}%
{\pdffilemoddate{#2}}>0%
{\immediate\write18{#3}}\fi%
}

\begin{document}

\title*{Geometry of Image Registration: The Diffeomorphism Group and Momentum Maps}
\titlerunning{Geometry of Image Registration}
\author{Martins Bruveris and Darryl D. Holm}
\institute{Martins Bruveris \at Institut de Math\'ematiques, EPFL, Lausanne 1015, Switzerland, \email{martins.bruveris@epfl.ch} 
\and Darryl D. Holm \at Department of Mathematics, Imperial College London, London SW7 2AZ, United Kingdom, \email{d.holm@imperial.ac.uk}}
%
%
\maketitle

\abstract{These lecture notes explain the geometry and discuss some of the analytical questions underlying image registration within the framework of \emph{large deformation diffeomorphic metric mapping} (LDDMM) used in computational anatomy.\footnote{Both authors gratefully acknowledge partial support by Advanced Grant 267382 from the European
Research Council.}}

\section{Introduction}\label{sec_intro}

The goal of computational anatomy is to model and study the variability of anatomical shape. The ideas of computational anatomy originate in the seminal book ``Growth and Form'' by D'Arcy Thompson \cite{Thompson1917}.

\begin{quotation}
In a very large part of morphology, our essential task lies in the comparison of related forms rather than in the precise definition of each; and the \emph{deformation} of a complicated figure may be a phenomenon easy of comprehension, though the figure itself have to be left unanalysed and undefined. This process of comparison [...] finds its solution in the elementary use of a certain method of the mathematician. This method is the Method of Coordinates, on which is based the Theory of Transformations. \cite[p1032]{Thompson1917}
\end{quotation}

More recently Grenander \cite{Grenander1981, Grenander1993, Grenander2007} generalized these ideas to encompass a diverse collection of real-world situations and formulated the principles of pattern theory. The following formulation is adapted from \cite{Mumford2010}:
\begin{enumerate}
\item
A wide variety of signals result from observing the world, all of which show patterns of many kinds. These patterns are caused by laws present in the world, but at least partially hidden from direct observation.
\item
Observations are affected by many variables that are not conveniently modelled deterministically because they are too complex or too dificult to observe.
\item
Patterns can be described as precise \emph{pure patterns} distorted and transfromed by a limited family of \emph{deformations}.
\end{enumerate}

To have a specific example in mind, we will consider computational neu\-ro\-ana\-tomy; i.e., the study of the form and shape of the brain \cite{Grenander1998}. The observations in this case are the diagnostic tools accessible to the clinician; of particular interest to us are noninvasive imaging techniques like computed tomography (CT), magnetic resonance imaging (MRI), functional MRI and diffusion tensor imaging (DTI). The hidden laws behind the observations are all the processes taking place at cellular, organ and environmental level, which together influence and form the anatomical shape of the brain.

To avoid having to model the brain from first principles, we observe instead that topologically all brains are very similar. If we take the MRI scans of two patients --- volumetric grey-scale images of two brains --- then we will be able to deform the contour surfaces of one image to approximately match the other. The study of shape and variability of brains within the framework of pattern theory, thus reduces to estimating the transformations that deform one brain image into an other. Given two images, the problem of finding this transformation is called the problem of \emph{image registration}. One then compares these transformations in order to infer information about shape and variability.

\runinhead{Outline of the notes.} The purpose of these lecture notes is to explain the geometry that underlies image registration within the LDDMM framework and to show how it is used in computational anatomy.
Section \ref{sec_intro} introduces the main objectives of computational anatomy. Section \ref{sec_geometry} explains the Lie group concepts and Riemannian geometry underlying the image registration problem. Finally, Sect. \ref{sec_analysis} sketches some of the analytical problems that arise in image registration within the LDDMM framework. The references are not exhaustive. Throughout, we rely on the fundamental texts \cite{Mumford2010,Younes2010}.

\runinhead{Image Registration with LDDMM.} Mathematically we model a volumetric grey-scale image $I$ as a function $I : \R^3 \to \R$ and we denote by $\mc F(\R^3)$ the collection of all such functions, subject to certain smoothness assumptions. We model transformations $\ph$ as smooth, invertible maps $\ph : \R^3 \to \R^3$ with smooth inverses. Such maps are called \emph{diffeomorphisms}. Invertibility ensures that tissue is not torn apart or collapsed to single points. The set of all transformations is denoted by $\on{Diff}(\R^3)$ and since it is closed under composition and taking the inverse, it forms a group, called the \emph{diffeomorphism group}. Deforming an image $I$ by the transformation $\ph$ corresponds to the change of coordinates $I \circ \ph^{-1}$. In the transformed image $I \circ\ph^{-1}$ the voxel $\ph(x)$ has the same grey-value as the voxel $x$ of the original image.

Given two images $I_0, I_1 \in \mc F(\R^3)$, the first apporach to the image registration problem would be to search for $\ph \in \on{Diff}(\R^3)$, such that $I_0\circ\ph^{-1} =I_1$. Two things can go wrong. First, such a $\ph$ may not exist and second, if it exists, it may not be unique. To address these problems, we can introduce a distance $d_1(\ph, \ps)$ on the set of transformations and a distance $d_2(I, J)$ on the set of images and search for the minimizer of
\begin{equation}
\label{imreg_dist}
\underset{\ph}{\on{argmin}}\; d_1(\on{Id}, \ph)^2 + \tfrac 1{\si^2} d_2(I_0\circ\ph, I_1)^2\;.
\end{equation}
The first term addresses the problem of uniqueness by ensuring that among all the transformations that deform $I_0$ into $I_1$ we pick the simplest one, by which we mean the one closest to the identity transformation. The second term allows us to compare images for which an exact solution to the registration problem does not exist, by requiring that the transformed image is close but not necessarily equal to $I_1$. Taken together \eqref{imreg_dist} represents a balance between finding a simple transformation and one that reproduces the given image. The parameter $\si^2$ controls this balance between simplicity or regularity of the transformation and the registration accuracy.

There are many possible definitions of a distance $d_1(\on{Id}, \ph)$ on the space of smooth invertible maps. We shall concentrate on the definition used in the \emph{large deformation diffeomorphic metric mapping} (LDDMM) approach \cite{Beg2005, Trouve2002, Miller2001, Trouve1998}, which generates the transformation $\ph = \ph_1$ as the flow of a time-dependent vector field. That is, $t \mapsto u_t$ is a solution to the flow equation
\[
\p_t \ph_t(x) = u_t(\ph_t(x)),\qquad \ph_0(x) = x\;.
\]
The distance $d_1(\on{Id}, \ph)$ is measured using a norm on the vector field $u_t$,
\[
d_1(\on{Id}, \ph)^2 = \inf_{\left\{ u_t \,:\, \ph = \ph_1\right\}}  \int_0^1 \left| u_t\right|^2 \ud t\;.
\]
Regarding the distance $d_2(I,J)$ on images, the simplest choice is the $L^2$-norm, i.e. $d_2(I, J)=|I-J|_{L^2(\R^3)}$, which will be used throughout these notes. The problem of image registration via LDDMM will form the basis of the following discussion.

\begin{definition}[Image Registration via LDDMM]
\label{def_lddmm}
Given two images $I_0, I_1 \in V$, find a time-dependent vector field $t \mapsto u_t \in \mf X(\R^3)$ that minimizes the energy
\begin{equation}
\label{lddmm_en}
E(u) = \frac 12 \int_0^1 \left| u_t \right|^2 \ud t + \frac{1}{2\si^2} \left| I_0\circ\ph_1^{-1} - I_1\right|^2_{L^2(\R^3)}\;,
\end{equation}
where $\ph_t \in \on{Diff}(\R^3)$ is the flow of $u_t$, i.e.
\[
\p_t \ph_t(x) = u_t(\ph_t(x)),\qquad \ph_0(x) = x\;,
\]
The vector field $t \mapsto u_t$ and the transformation $\ph_1$ are the solutions of the image registration problem.
\end{definition}

This is not the only possible approach to image registration. In fact, a large literature about image registration exists. An overview of the available methods can be found, e.g. in \cite{Holden2008, Klein2009}. The LDDMM method, whilst being computationally more expensive than others, is among the most accurate \cite{Ashburner2011} registration methods. Here we will concentrate on the geometric structure of the LDDMM solutions. In particular, we will sketch some applications in which the geometry behind LDDMM helps illuminate relationships between anatomical shape and neurological functions.

Data structures other than images can be registered within the LDDMM framework. These include landmarks \cite{Joshi2000, Glaunes2004}, curves \cite{Glaunes2008, Durrleman2009}, surfaces \cite{Valliant2005}, tensor fields \cite{Alexander2001, Cao2005} or functional data on a manifold \cite{Miller2009, Qiu2006}. In fact the abstract formulation of LDDMM in Sect. \ref{sec_geometry} encompasses all these examples. Instead of the $L^2$-norm one can use other similarity metrics to measure the distance between images, e.g. mutual information \cite{Maes1997}. The biggest departure from LDDMM would be to change the way diffeomorphisms are generated. Possible approaches are stationary vector fields \cite{Ashburner2007}, free-form deformations \cite{Rueckert1999}, only affine transformations \cite{Jenkinson2001} or demons \cite{Thirion1998, Vercauteren2008}. Common to all these methods however is the loss of geometric structure.

\runinhead{Anatomical Shape and Function.} Alzheimer's disease (AD) is a neurodegenerative disease and is the most frequent type of dementia in the elderly \cite{Ferri2006}. Related to AD is mild cognitive impairment (MCI), an intermediate cognitive state between healthy ageing and dementia. Although most patients who develop AD are first diagnosed with MCI, not all of those with MCI will develop AD. There is considerable variability among the prognoses of patients with MCI: some develop into AD, while others remain stable, revert back to normal cognitive status or develop other forms of dementia. It is therefore of interest to find methods of predicting the prognosis of patients with MCI. One approach is to look for manifestations of AD and MCI in the anatomical shape and to find connections between anatomical shape and clinical measures of cognitive status that are used to diagnose and distinguish between AD, MCI and normal cognitive state (NCS) \cite{Weiner2012}.

\subruninhead{Alzheimer's Disease and the Shape of Subcortical Structures.}  In \cite{Qiu2009} a population $(I_j)_{1\leq j \leq 383}$ of 383 subjects, both healthy and diseased, was registered to a common template $I_{\mathrm{templ}}$; i.e., for each pair $I_j$, $I_{\mathrm{templ}}$ a deformation $\ph_j$, satisfying $I_{\mathrm{templ}}\circ\ph_j^{-1} \approx I_j$, was computed by solving the registration problem in Def. \ref{def_lddmm}. Seven subcortical structures $S^1,\dots,S^7$ were extracted from each image and the log-Jacobian $f^k_j = \log \left(\det D\ph_j\right)|_{\p S^k}$ of the estimated transformation $\ph_j$, restricted to the boundary of the structure $S^k$, was used to measure the shape variation with respect to the template. These maps $f^k_j$ were called ``surface deformation maps''. After performing principal component analysis on these maps followed by linear regression with the diagnosis (AD, MCI or NCS), it was found for example that AD and MCI, when compared to NCS is associated with a pronounced surface inward deformation in areas of the amygdala and the hippocampus and with a simultaneous outward deformation in the body and inferior lateral ventricles. These results are in agreement with previous neuroimaging findings and show that LDDMM can be used to highlight local shape variations related to AD.

\begin{figure}[ht]
  \centering
  \def\svgwidth{0.6\textwidth}
\executeiffilenewer{parallel_transport.svg}{parallel_transport.pdf}%
{inkscape -z -D --file=parallel_transport.svg %
--export-pdf=parallel_transport.pdf --export-latex}%
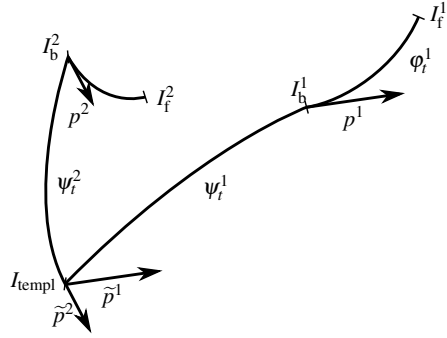%

  \caption{The use of parallel transport in a longitudinal study AD. The baseline scan $I^1_b$ is registered to the follow-up scan $I^1_f$ via $\ph^1_t$ and the baseline scan is registered to the template image $I_{\mathrm{templ}}$ via $\ps^1_t$. The deformation $\ph^1$ is encoded in the initial momentum $p^1$, which is parallel transported along the path $\ps^1_t$ to $I_{\mathrm{templ.}}$ to obtain $\wt p^1$. In this way the changed between baseline and follow-up scans can be compared across a population of patients.}
  \label{fig_parellel_transport}
\end{figure}

\runinhead{Analysis of Longitudinal Data.} A more accurate assessment of disease states can be obtained by comparing two different scans of one patient, taken at two different times. Let $I_{\mathrm b}^j$ and $I_{\mathrm f}^j$ denote the baseline scan and the follow-up scan taken a few years later of the $j$-th patient respectively. Registering $I_{\mathrm b}^j$ to $I_{\mathrm f}^j$ via LDDMM computes a transformation $\ph_j$ such that $I_{\mathrm b}^j\circ\ph^{-1}_j \approx I_{\mathrm f}^j$, and also its generating vector field $u_t^j$. Sect. \ref{sec_geometry} shows that the entire vector field $u_t^j$ can be recovered from its value at $t=0$ via the \emph{Euler-Poincar\'e equation on the diffeomorphism group}, also called \emph{EPDiff} and introduced in equation \eqref{eq_epdiff}. This means that the initial vector field $u_0^j$ can be determined from the initial \emph{deformation momentum}, $p^j$. Thus, the shape differences between $I^j_{\mathrm b}$ and $I_{\mathrm f}^j$ are encoded in the deformation momentum $p^j$. 

To compare the deformation momenta $p^j$, $j=1,2,\dots$, for a set of different patients all baseline scans $I_{\mathrm b}^j$ are registered in a second step to a common template $I_{\mathrm{templ}}$. Then it is necessary to transport each of the deformation momenta $p^j$ from $I_{\mathrm b}^j$ to the common template and thereby obtain the corresponding $\wt p^j$. Thus it is possible to compare the evolution between the baseline and the follow-up scans, by its nature a very nonlinear object, by comparing the computed momenta $\wt p^j$, which are elements of a vector space. 

Regarding the transport operation several methods have been proposed. From a geometrical point of view, parallel transport from Riemannian geometry is the most natural operation and this has been used in \cite{Qiu2009b, Younes2007, Younes2008}. Since computing the parallel transport of the momentum along geodesics is numerically quite challenging, a first-order approximation called Schild's ladder was proposed in \cite{Lorenzi2011} as an alternative. Other methods that depend only on the end-deformation and not on the whole geodesic path were considered and compared in \cite{Fiot2012}. From the viewpoint of applications, there is, as of now, no consensus about which is the most appropriate method for the transport of deformation momenta.

\subruninhead{Longitudinal Study of the Shape of Hippocampi.} Parallel transport was used in \cite{Qiu2008} as the transport method to compare deformations of the hippocampus in subjects with early AD and healthy controls across a time span of two years. It was shown that the conversion from normal cognitive function to early AD in the time span between the baseline scan and the follow-up scan is associated with an inward deformation of the hippocampal tail. Subjects who were already diagnosed with AD at the time of the baseline scan on the other hand exhibited an inward deformation of the whole hippocampal body.

\runinhead{Propagation of Anatomical Information.} Registering two images $I_0$ and $I_1$ via a transformation $\ph$ gives us a voxel-to-voxel correspondence between these two images. Assuming that we are given a manual segmentation of the template image $I_0$, in which some or all voxels of $I_0$ are assigned membership to a labelled anatomical structure, we can propagate this segmentation via $\ph$ to the image $I_1$. This is the idea of registration-based or atlas-based segmentation, see \cite{Collins1992, Gee1993, Miller1993}. To remove the bias inherent in the choice of the template $I_0$, these multi-template registration techniques replace $I_0$ with a collection $(I^j_{\mathrm{templ}})_{j=1,\dots,N}$ of several manually segmented images. Each template $I^j_{\mathrm{templ}}$ is registered to $I_1$ and the transformation is used to propagate the segmentation of $I^j_{\mathrm{templ}}$ to $I_1$. Now there are $N$ potentially contradicting segmentations of $I_1$, that have to be combined using some classifier \emph{fusion} technique, such as majority voting \cite{Artaechevarria2009}, Bayesian modelling \cite{Warfield2004} or Markov random fields \cite{Fischl2002}. For subcortical structures of the brain, this sort of atlas-based segmentation was shown to outperform other methods \cite{Babalola2009}. It is possible to use atlas-based registration with a variety of image registration methods. However, a study that involved segmenting brain scans of mice \cite{Bai2012} has shown that the choice of the registration method is more important than the choice of fusion method. Thus in applications where accuracy is important, LDDMM may be preferred, despite having higher computational cost than some other registration methods.

\subruninhead{Automatic Labelling via Ontologies.} In the same spirit, Steinert-Threlkeld
 et al. \cite{SteinertThrelkeld2012} combined an ex-vivo scan of the left ventricle, manually parcellated and labelled, with the LDDMM registration method and an ontology query language to allow the medical practitioner to obtain quantitative and qualitative answers to questions like: ``In which region was significant tissue volume expansion observed between systole and diastole?'' and ``What was the average rate of expansion per region of interest?'' The ability to automatically answer these questions is a key step toward automating the diagnostic process.

\subruninhead{Patient-Specific Models for Atrial Fibrillation \cite{McDowell2012}.} Atrial fibrillation is a cardiac arrhythmia, characterized by the irregular propagation of electrocardial waves across the atrium. Advances in late-gadolinium enhanced MRI, allows the in-vivo localization of fibrotic tissue in the atrium, by diffusion tensor imaging (DTI). Although the DTI approach does not yet have the necessary resolution to determine the fibre orientation of the muscle fibres in-vivo, there already exist atlases with information about fibre orientation, obtained ex-vivo. Image registration can be used to propagate the fibre orientations from the atlas to the patient and thus obtain a patient-specific model of the atrium that includes both locations of fibrotic tissue and orientation of the muscle fibres. The resulting model can then be used to simulate the propagation of electrocardial waves and to predict the occurrence of arrhythmia in the patient's atrium.

\runinhead{Other Applications.} We cannot hope to give an exhaustive description of all the applications of LDDMM and its associated geometry to computational medicine in these notes. Among the omitted topics are: estimating the dimensionality of the anatomical shape variations \cite{Qiu2012}; generalizing geodesic regression to the anatomical shape manifold and computing the mean aging process of the brain across a population \cite{DAvis2010}; the use of parallel transport not only for longitudinal studies, but also to characterize the left-right asymetry of subcortical structures \cite{Qiu2009c}; applications to other diseases like schizophrenia \cite{Qiu2007} or cerebral palsy \cite{Faria2011}; addition of functional data to anatomical shapes \cite{Miller2009}. There are also applications outside the medical field to the study of variations of cell shape \cite{Rohde2008} and to construct generative models for cells \cite{Peng2009}.

\section{Geometry of Matching Problems}
\label{sec_geometry}

In order to better see the geometric properties of image registration with LDDMM, we will first formulate an abstract version of it. As we study this abstract problem, we will at each step show how it relates to the concrete example of image registration.

\runinhead{Abstract Formulation.} In the spirit of pattern theory we can formulate image registration as follows: a group of transformations acts on a space of objects and we are searching for the transformation that deforms a template object to a target object. The presentation here follows \cite{Bruveris2011}.

Let us model the group of transformations by a Lie group $G$ and the space of objects by a vector space $V$. We will in this section assume that both $G$ and $V$ are finite-dimensional in order to avoid questions about topologies, smoothness and dual spaces that arise when dealing with infinite-dimensional spaces. The process of deforming objects $I \in V$ by transformations $g \in G$ is modelled by a smooth map
\[
\ell : G \x V \to V,\qquad (g, I) \mapsto g.I\;.
\]
Note that $g.I$ is simply a notation for $\ell(g, I)$, i.e. the object $I$ transformed under $g$. Let $e \in G$ denote the neutral element of the group. We require $\ell$ to satisfy the following axioms
\begin{itemize}
\item $\ell(e, I) = I$ or $e.I = I$ for $I \in V$ and
\item $\ell(g, \ell(h, I)) = \ell(gh, I)$ or $g.(h.I) = (gh).I$ for $g, h \in G$ and $I \in V$,
\end{itemize}
The first axiom tells us that the identity transformation doesn't change the object while the second is an associativity axiom and allows us to write simply $gh.I$ for either $g.(h.I)$ or $(gh).I$. Such a map $\ell$ is called a {\it left action of the group $G$ on the vector space $V$}. An in-depth treatment of group actions, beyond what we will need for our purposes, can be found, e.g. in \cite{Michor2008b}.

\begin{example}
Consider the rotation group $SO(3)$ and the vector space $\R^3$. The action
\[
\ell : SO(3) \x \R^3 \to \R^3,\qquad (R,x) \mapsto Rx
\]
is given by matrix multiplication. The rules of matrix algebra imply that this is indeed a left action.
\end{example}

To generate deformations and to measure their ``size'' or ``energy'', we use the linearization of the Lie group $G$. The {\it Lie algebra $\mf g$ of $G$} is the tangent space of $G$ at the identity, i.e. $\mf g = T_e G$. Intuitively $\mf g$ consists of ``infinitesimal deformations''. The following points of view are equivalent:
\begin{itemize}
\item
Given a smooth curve $t \mapsto u_t \in \mf g$ of infinitesimal deformations, there exists a curve $t \mapsto g_t \in G$ in the group, which is the solution of the differential equation
\[
\p_t g_t = u_t g_t,\qquad g_0 = e\;.
\]
The curve $g_t$ is called the {\it flow} or {\it integral curve of $u_t$}.
\item
Given a smooth curve $t \mapsto g_t \in G$ of deformations, its velocity is $\p_t g_t \in T_{g_t} G$ and it defines a curve $u_t := (\p_t g_t) g_t^{-1} \in T_e G$ of infinitesimal deformations. The curve $u_t$ is called the {\it right-trivialized velocity of $g_t$}.
\end{itemize}

To complete the modelling of the matching problem we assume that both the Lie algebra $\mf g$ and the space $V$ of objects are endowed with inner products $\langle .,. \rangle_{\mf g}$ and $\langle .,. \rangle_V$ respectively. The kinetic energy of a curve $g_t$ of deformations is measured via its right-trivialized velocity
\[
E_{\mathrm{KE}}(u) = \tfrac 12 \int_0^1 \left| u_t \right|_{\mf g}^2 \ud t\;,
\]
where $|u|_{\mf g} = \sqrt{\langle u, u \rangle_{\mf g}}$ is the norm induced by the inner product. The inner product on $V$ will be used to measure the distance between objects. The matching problem can now we stated as follows.

\begin{definition}[Abstract Registration Problem]
\label{abstract_lddmm}
Given two objects $I_0, I_1 \in V$ find a curve $t \mapsto u_t \in \mf g$ that minimizes the energy
\begin{equation}
\label{abs_en}
E(u) = \tfrac 12 \int_0^1 \left| u_t \right|_{\mf g}^2 \ud t + \tfrac{1}{2\si^2} \left| g_1.I_0 - I_1\right|^2_V\;,
\end{equation}
where $g_1 \in G$ is the endpoint of the flow of $u_t$, i.e.
\[
\p_t g_t = u_t g_t,\qquad g_0 = e\;.
\]
The transformation $g_1$ then matches $I_0$ to $I_1$.
\end{definition}

We defer questions about existence of minimizers to Sect.~\ref{sec_analysis}. Our goal now is to study properties of the minimizing curves $u_t$. In particular, we want to see which properties of the minimizer are fixed by the group and what features of it are affected by the choice of the space of objects. We assume all objects to be sufficiently smooth. Thus, minima of $E$ are also critical points; so we will be interested in calculating the derivative $DE(u)$. In order to do that we need some more tools from geometry.

\runinhead{The Adjoint Action.} 
On the group $G$ we fix an element $g \in G$ and consider the map
\[
\on{conj}_g : G \to G,\qquad \on{conj}_g(h) = ghg^{-1}\;,
\]
called \emph{conjugation}. It satisfies $\on{conj}_g(e) = e$ and we denote its tangent map by
\[
\on{Ad}_g := T_e \on{conj}_g : T_e G \to T_e G\;.
\]
This map is called the {\it adjoint representation} of G. The following properties of $\on{conj}$ can be easily verified,
\begin{align*}
\on{conj}_g \circ \on{conj}_h &= \on{conj}_{gh} \\
\on{conj}_{g^{-1}} &= \left(\on{conj}_g\right)^{-1}.
\end{align*}
These properties imply the following differential versions,
\begin{align*}
\on{Ad}_g \on{Ad}_h &= \on{Ad}_{gh} \\
\on{Ad}_{g^{-1}} &= \left(\on{Ad}_g\right)^{-1}.
\end{align*}
Considered as a map of both variables, the operation $\on{Ad} : G \x \mf g \to \mf g$ defines a left action of $G$ on its Lie algebra $\mf g$. We also see that $\on{Ad}$ is a group homomorphism $\on{Ad} : G \to GL(\mf g)$ between $G$ and the group $GL(\mf g)$ of invertible linear maps on $\mf g$. This property is the reason for the name adjoint representation.

\runinhead{The Coadjoint Action.}
Again keeping $g \in G$ fixed we consider the linear map $\on{Ad}_g : \mf g \to \mf g$.  This map has a transpose $\on{Ad}^\ast : \mf g^\ast \to \mf g^\ast$ in the sense of linear algebra, defined via
\[
\left\langle \on{Ad}_g^\ast \mu, u \right\rangle_{\mf g^\ast \x \mf g} = 
\left\langle \mu, \on{Ad}_g u \right\rangle_{\mf g^\ast \x \mf g}\;,
\]
for $\mu \in \mf g^\ast$ and $u \in \mf g$. This map $\on{Ad}^\ast$ is called the {\it coadjoint representation} of $G$. Similarly to $\on{Ad}$ it satisfies
\begin{align*}
\on{Ad}^\ast_g \on{Ad}^\ast_h &= \on{Ad}^\ast_{hg} \\
\on{Ad}^\ast_{g^{-1}} &= \left(\on{Ad}^\ast_g\right)^{-1}\;.
\end{align*}
Considered as a map of both variables, the map $\on{Ad}^\ast : G \x \mf g^\ast \to \mf g^\ast$ defines a right action of $G$ on $\mf g^\ast$. It is not a left action, because in the associativity rule the order of the multiplication is changed. To make it into a left action we can consider the map $(g,\mu) \mapsto \on{Ad}^\ast_{g^{-1}} \mu$. The name coadjoint representation stems from the way of looking at $\on{Ad}^\ast$ as a group antihomomorphism $\on{Ad}^\ast : G \to GL(\mf g^\ast)$.

\runinhead{Variations of the Flow.}
Why is this interlude necessary? In order to differentiate the term $\left| g_1.I_0 - I_1\right|^2_V$ in \eqref{abs_en} we need to know how to differentiate $g_1$ with respect to $u_t$, since $g_1$ is defined as the flow
\[
\p_t g_t = u_t g_t,\qquad g_0 = e\;,
\]
of $u_t$ at time $t=1$. This is given in the following lemma, the proof of which is adapted from \cite{Vialard2009} and \cite{Beg2005}.

\begin{lemma} 
\label{delta_u_right}
Let $t \mapsto u_t \in \mf g$ be a smooth curve and $(\ep, t) \mapsto u^\ep_t$ a smooth variation of this curve. Denote by $\de u_t := \p_\ep|_{\ep=0} \left(u_t^\ep\right)$ an infinitesimal variation of $u_t$. Then
\[
\de g_t := \p_\ep|_{\ep=0} \left(g_t^\ep\right) = g_t \int_0^t \on{Ad}_{g_s^{-1}} \de u_s \ud s\;.
\]
\end{lemma}
\begin{proof} For all $\ep$ we have
\[
\p_t g_t^\ep =u_t^\ep g^{\ep}_{t},\qquad g_{0}^{\ep}=e\;.
\]
Taking the $\varepsilon$-derivative of this equality yields the ODE
\[
\p_t \p_\ep|_{\ep=0} \left( g_t^\ep \right) = \de u_t g_t + u_t \de g_t \;,
\]
and so we obtain
\begin{align*}
\p_t \left(g_t^{-1} \de g_t \right) 
&= -g_t^{-1} u_t g_t g_t ^{-1} \de g_t + g_t^{-1}\left( \de u_t g_t + u_t \de g_t \right) \\
&= g_t^{-1} \de u_t g_t \\
&= \on{Ad}_{g_t^{-1}} \de u_t\;.
\end{align*}
Now we integrate both sides from $0$ to $t$ and multiply by $g_t$ from the left to obtain
\[
\de g_t = g_t \int_0^t \on{Ad}_{g_s^{-1}} \de u_s \ud s\;,
\]
as required.
\qed
\end{proof}

The second tool we will need to compute the derivative $DE(u)$ is a map that describes the relation between the group $G$ and the space $V$ it acts upon. This map is called the \emph{momentum map}.

\runinhead{The Momentum Map.}

Starting with the action $\ell: G \x V \to V$ of a Lie group $G$ on a vector space $V$, we fix $I \in V$ and consider the map $\ell^I : G \to V$ given by $\ell^I(g) = \ell(g,I)$. The derivative of this map at $e \in G$ is $T_e \ell^I : \mf g \to T_I V$ and it may be interpreted, if we allow $I$ to vary, as a vector field on $V$, i.e. now keep $u \in \mf g$ fixed and consider
\begin{equation}
\label{fund_vec}
\ze_u : V \to TV,\qquad I \mapsto T_e \ell^I.u\;.
\end{equation}
Thus $\ze : \mf g \to \mf X(V)$ assigns to each Lie algebra element $u$ a vector field $\ze_u$ on $V$. These are called the \emph{fundamental vector fields} of the $G$-action. We will also use the notation $\ze_u(I) = u.I$.

The tangent bundle $TV$ of $V$ can be identified via $TV \cong V \x V$ with two copies of $V$, the first containing basepoints and the second the tangent vectors. Similarly we can identify the cotangent bundle $T^\ast V$ with the product $T^\ast V \cong V \x V^\ast$. Now take an element $(I, \pi) \in T^\ast V$. The pairing
\[
\left\langle \pi, \ze_u(I)\right\rangle_{V^\ast\x V}\;,
\]
is linear in $u \in \mf g$ as can be seen from \eqref{fund_vec} and thus $u \mapsto \left\langle \pi, \ze_u(I)\right\rangle_{V^\ast\x V}$ is a linear form on $\mf g$ or equivalently an element of $\mf g^\ast$. Denote this element by $I \diamond \pi$. The defining equation for $I \diamond \pi \in \mf g^\ast$ is
\[
\left\langle I \diamond \pi, u \right\rangle_{\mf g^\ast \x g} = \left\langle \pi, \ze_u(I) \right\rangle_{V^\ast \x V}\;,
\]
and $\diamond$ is a map $\diamond : T^\ast V \to \mf g^\ast$, called the \emph{momentum map} of the cotangent lifted action of $G$ on $T^\ast V$. We shall explain the action of $G$ on $T^\ast V$ in the following paragraph.

\runinhead{Momentum Maps in Geometric Mechanics.}
In geometric mechanics, momentum maps generalize the notions of linear and angular momenta. For a mechanical system, whose configuration space is a manifold $M$ acted on by a Lie group $G$, the momentum map $\diamond: T^\ast M \to \mf g^\ast$ assigns to each element of the phase space $T^\ast M$ a \emph{generalized momentum} $I \diamond \pi$ in the dual $\mf g^\ast$ of the Lie algebra. For example, the momentum map for spatial translations is the linear momentum, and for rotations it is the angular momentum.

One important feature of the momentum map in geometric mechanics is due to Noe\-ther's theorem. Noether's theorem states that if the Hamiltonian of the system under consideration is invariant under the action of $G$, then the generalized momentum $I \diamond \pi$ is a constant of motion. This theorem enables one generate conservation laws from symmetries. See \cite{Holm2009b, Marsden1999} for more details on momentum maps and geometric mechanics.

\runinhead{Tangent and Cotangent Lifted Actions.}

The action of a Lie group $G$ on the vector space $V$ is a map $\ell: G \x V \to V$. Fixing an element $g \in G$ we obtain a map $\ell_g : V \to V$, which we can differentiate to obtain $T \ell_g : TV \to TV$.
It can be checked that the map of both variables
\[
T_2 \ell : G \x TV \to TV\;,
\]
is a left action of $G$ on the space $TV$. Here $T_2 \ell$ denotes the derivative of $\ell$ with respect to the second variable. The map $T \ell_g : V \x V \to V \x V$, being a derivative, is linear in the second variable, i.e. for each $I$ the map 
\[
T_I \ell_g : V  \cong T_I V \to T_{g.I} V \cong V\;,
\]
is linear and thus has a transpose
\[
T_I^\ast \ell_g : V^\ast  \cong T^\ast_{g.I} V \to T^\ast_{I} V \cong V^\ast\;.
\]
This allows us to define the \emph{cotangent lifted action} of $G$ on the cotangent bundle $T^\ast V \cong V \x V^\ast$ via
\[
g.(I, \pi) = \left( \ell(g, I), T_{g.I}^\ast \ell_{g^{-1}}. \pi\right)\;,
\]
for $(I, \pi) \in T^\ast V$. Note that the presence of the inverse makes this a left action. The following lemma shows that the momentum map is \emph{equivariant} with respect to the cotangent lifted action.

\begin{lemma}
For $g \in G$, $u \in \mf g$, $I \in V$ and $\pi \in V^\ast$ we have
\begin{itemize}
\item
$g. \ze_u(g^{-1}.I) = \ze_{\on{Ad}_g u}(I)$, and
\item
$g.I \diamond g.\pi = \on{Ad}^\ast_{g^{-1}} \left(I \diamond \pi\right)$.
\end{itemize}
\end{lemma}

\begin{proof}
First note that $g.\ze_u(g^{-1}.I)$ is a slightly informal way to denote $g$ acting on $\ze_u(g^{-1}.I)$ via the cotangent lifted action; i.e.,
\[
g.\ze_u(g^{-1}.I) = T_{g^{-1}.I} \ell_g. \ze_u(g^{-1}.I)\;.
\]
To prove the first identity take a curve $h(t) \in G$ with $h(0) = e$ and $\p_t g|_{t=0}=u$. Via associativity, we have
\[
\ell(g, h(t).g^{-1}.I) = \ell(gh(t)g^{-1},I)\,,
\]
and by differentiating this identity we obtain
\begin{align*}
T_{g^{-1}.I} \ell_g.\ze_u(g^{-1}.I) &= T_e \ell^I.\on{Ad}_gu \\
g.\ze_u(g^{-1}.I) &= \ze_{\on{Ad}_g u}(I)\;.
\end{align*}
For the second identity note that $g.I \diamond g.\pi$ is a short way of writing
\[
g.I \diamond g.\pi = \diamond\,(g.(I,\pi)) = \diamond\,(g.I, T^\ast_{g.I} \ell_{g^{-1}}.\pi) = g.I \diamond T^\ast_{g.I} \ell_{g^{-1}}.\pi\,.
\]
Now take any $u \in \mf g$ and consider the pairing
\begin{multline*}
\langle g.I \diamond g.\pi, u \rangle_{\mf g^\ast \x \mf g} =
\langle T^\ast_{g.I} \ell_{g^{-1}}.\pi, \ze_u(g.I) \rangle_{V^\ast \x V} = \\ =
\langle \pi, T_{g.I} \ell_{g^{-1}}.\ze_u(g.I) \rangle_{V^\ast \x V} = 
\langle \pi, \ze_{\on{Ad}_{g^{-1}} u}(I) \rangle_{V^\ast \x V} = \\ =
\langle \on{Ad}^\ast_{g^{-1}}\left(I\diamond \pi\right), u \rangle_{\mf g^\ast \x \mf g}\,.
\end{multline*}
This concludes the proof.
\qed
\end{proof}

\runinhead{The $\flat$-map.}
The final piece of notation is the $\flat$-map of a vector space associated to an inner product. On the vector space $V$ the $\flat$-map is defined as
\[
\flat : V \to V^\ast,\qquad \langle u^\flat, v \rangle_{V^\ast\x V} = \langle u, v \rangle\;,
\]
where the pairing on the left side is the canonical pairing between $V^\ast$ and $V$ and on the right side we have the inner product $\langle.,,\rangle_V$. Each inner product gives rise to a $\flat$-map and we have two of them in our framework, one on $\mf g$ and one on $V$. As there is no risk of confusion between them, we will use the same notation for both. Inspired by their appearance in physics, the elements $u \in \mf g$ are called \emph{velocities} while the dual objects $u^\flat \in \mf g^\ast$ are called \emph{momenta}.

\runinhead{Derivative of the Matching Energy.}
We now have assembled all of the tools we need to calculate the derivative $DE(u)$.

\begin{theorem}
\label{lddmm_DE}
Consider the matching energy
\[
E(u) = \frac 12 \int_0^1 \left| u_t \right|_{\mf g}^2 \ud t + \frac 1{2\si^2} \left| g_1.I_0 - I_1 \right|_V^2\;.
\]
Its derivative is given by
\begin{equation}
\label{eq_DE}
DE(u)(t) = u_t^\flat + g_t I_0 \diamond g_t g_1^{-1} \pi\;,
\end{equation}
with $\pi = \frac 1{\si^2} (g_1I_0 - I_1)^\flat \in V^\ast \cong T^\ast_{g_1.I_0} V$.
\end{theorem}

\begin{proof}
The derivative is a curve $t \mapsto DE(u)(t) \in \mf g^\ast$ and the pairing between $DE(u)$ and a variation $\de u$ is given by
\[
\left\langle DE(u), \de u \right\rangle = \int_0^1 \left\langle DE(u)(t), \de u_t \right\rangle_{\mf g^\ast \x \mf g} \ud t\;.
\]
From
\[
\left\langle D\left( \frac 12 \int_0^1 \left|u_t \right|_{\mf g}^2 \ud t \right) , \de u \right\rangle = \int_0^1 \left\langle u_t, \de u_t \right\rangle_{\mf g} \ud t
= \int_0^1 \left\langle u_t^\flat, \de u_t \right\rangle_{\mf g^\ast \x \mf g} \ud t\;,
\]
we see that the derivative of the kinetic energy part is simply $u_t^\flat$. Now for the matching term,
\[
\left\langle D\left( \frac{1}{2\si^2} \left| g_1.I_0 - I_1 \right|_V^2 \right), \de u \right\rangle = \frac 1{\si^2} \left\langle g_1.I_0 - I_1, \de g_1.I_0 \right\rangle_{V^\ast \x V} = \left\langle \pi, \de g_1.I_0 \right\rangle_{V^\ast \x V}\;.
\]
We apply Lem. \ref{delta_u_right} to express $\de g_1$ via $\de u$ and the we use adjoint operations to isolate $\de u$. Consequently, we find
\begin{align*}
\left\langle \pi, \de g_1.I_0 \right\rangle_{V^\ast \x V} &=
\left\langle \pi, g_1.\int_0^1 \on{Ad}_{g_t^{-1}} \de u_t \ud t.I_0 \right\rangle_{V^\ast \x V} \\
&= \int_0^1 \left\langle g_1^{-1}.\pi, \left( \on{Ad}_{g_t^{-1}} \de u_t \right).I_0 \right \rangle_{V^\ast \x V} \ud t \\
&= \int_0^1 \left\langle I_0 \diamond g_1^{-1}.\pi, \on{Ad}_{g_t^{-1}} \de u_t \right \rangle_{\mf g^\ast \x \mf g} \ud t \\
&= \int_0^1 \left\langle \on{Ad}_{g_t^{-1}}^\ast \left( I_0 \diamond g_1^{-1}.\pi \right), \de u_t \right \rangle_{\mf g^\ast \x \mf g} \ud t \\
&= \int_0^1 \left\langle g_t.I_0 \diamond g_tg_1^{-1}.\pi, \de u_t \right \rangle_{\mf g^\ast \x \mf g} \ud t\;.
\end{align*}
And thus we obtain the result. \qed
\end{proof}

\runinhead{Image Matching.}
In image matching the group of transformations is taken to be the group $\on{Diff}(\R^3)$ of diffeomorphisms of $\R^3$, i.e., smooth invertible maps $\ph : \R^3 \to \R^3$ with smooth inverses. The space of objects is $\mc F(\R^3)$, the space of real-valued smooth functions on $\R^3$, and the action is given by
\[
\ell : \on{Diff}(\R^3) \x \mc F(\R^3) \to \mc F(\R^3), \qquad (\ph, I) \mapsto I\circ\ph^{-1}\;.
\]
Due to the inverse in the definition, the voxel $\ph(x)$ of the transformed image has the same grey-value as the voxel $x$ of the original image. We will postpone the discussion of analytical aspects of $\on{Diff}(\R^3)$ to Sect. \ref{sec_analysis} and for now assume all objects are sufficiently smooth for the necessary operations.

\begin{remark}[Convenient Calculus]
The discussion here can be made rigorous by considering the group
\[
\on{Diff}_{H^\infty}(\R^3) = \left\{ \ph\,:\, \on{Id}-\ph \in H^\infty(\R^3) \right\}
\]
of all diffeomorphisms $\ph$, such that $\on{Id}-\ph$ lies in the intersection $H^\infty(\R^3)$ of all Sobolev spaces. The group $\on{Diff}_{H^\infty}(\R^3)$ is a smooth regular Fr\'echet-Lie group. For the space of images we can take either $\mc F(\R^3) = H^\infty(\R^3)$ functions with square-integrable derivatives or $\mc F(\R^3) = C_c^\infty(\R^3)$ compactly supported functions. Then the action
$\ell : \on{Diff}_c(\R^3) \x C_c^\infty(\R^3) \to C_c^\infty F(\R^3)$
is smooth in the sense of convenient calculus \cite{Michor1997} and all the operations described below can be interpreted in that framework. See \cite{Michor2012b_preprint} for details on diffeomorphism groups with other decay properties.
\end{remark}

The Lie algebra of $\on{Diff}(\R^3)$ is $\mf X(\R^3)$, the space of vector fields on $\R^3$. Given a time-dependent vector field $t \mapsto u_t \in \mf X(\R^3)$ its flow is defined by the differential equation
\[
\p_t \ph_t(x) = u_t\left(\ph_t(x)\right), \qquad \ph_0(x) = x, \qquad x \in \R^3 \;.
\]
Let us assume that we are given a norm on $\mf X(\R^3)$, defined via a positive, self-adjoint differential operator $L$ as follows,
\begin{equation}
\label{L_norm}
\left\langle u, v \right\rangle_L = \int_{\R^3} u(x) \cdot Lv(x) \ud x \;.
\end{equation}
For example the $H^1$-norm
\[
\left\langle u, v \right\rangle_{H^1} = \int_{\R^3} u(x) \cdot v(x) + \al^2 \sum_{i=1}^3 \nabla u^i(x) \cdot \nabla v^i(x) \ud x\;,
\]
can be defined via the operator $Lu = u - \al^2 \De u$, where the Laplace operator is understood to act componentwise on $u$. 

The dual space of $\mf X(\R^3)$ is the space of distributions. We consider only the smooth dual, that is the space $\mf X(\R^3)^\ast := \left\{ Lu \,:\, u \in \mf X(\R^3)\right\}$ generated by the $\flat$-map. As the duality pairing between $\mf X^\ast(\R^3)$ and $\mf X(\R^3)$ we choose the $L^2$-pairing, i.e.
\[
\left\langle \al, u \right\rangle_{\mf X(\R^3)^\ast \x \mf X(\R^3)} = \int_{\R^3} \al(x) \cdot u(x) \ud x\;.
\]
Thus we see that the $\flat$-map of the $\langle .,.\rangle_L$-inner product is given by $u^\flat = Lu$.

On the space of images we use the $L^2$-inner product $\langle I, J \rangle_{L^2} = \int_{\R^3} I(x)J(x) \ud x$. Again we don't look at the whole dual space, but only at the subspace generated by functionals of the form $I \mapsto \int_{\R^3} \pi I \ud x$ with $\pi \in \mc F(\R^3)$. Thus the canonical pairing is given by
\[
\langle \pi, I \rangle_{\mc F(\R^3)^\ast \x \mc F(\R^3)} = \int_{\R^3} \pi(x) I(x) \ud x\;.
\]
The $\flat$-map in this case is the identity, $I^\flat = I$. However the distinction between $\mc F(\R^3)$ and its dual $\mc F(\R^3)^\ast$ is still important, because $\on{Diff}(\R^3)$ will act differently on the spaces.

The infinitesimal action of $u \in \mf X(\R^3)$ on $I \in \mc F(\R^3)$ can be computed via
\[
\ze_u(I) = \p_t|_{t=0} \ph_t.I\;,
\]
where $t \mapsto \ph_t$ is a curve with $\ph_0 = \on{Id}$ and $\p_t|_{t=0} \ph_t = u$. Then
\[
\ze_u(I) = \p_t|_{t=0} \left(I \circ \ph_t^{-1}\right) = -\nabla I \cdot u\;.
\]
This allows us to compute the momentum map
\begin{align*}
\left\langle I \diamond \pi, u \right\rangle_{\mf X(\R^3)^\ast \x \mf X(\R^3)}
&= \left\langle \pi, \ze_u(I) \right\rangle_{\mc F(\R^3)^\ast \x \mc F(\R^3)} \\
&= -\int_{\R^3} \pi(x) \nabla I(x) \cdot u(x) \ud x \\
&= \left\langle -\pi \nabla I, u \right\rangle_{\mf X(\R^3)^\ast \x \mf X(\R^3)}\;.
\end{align*}
Thus, in this case, $I \diamond \pi = -\pi \nabla I$.

The last pieces of the geometrical framework are the lifted tangent and cotangent actions. The action of $\on{Diff}(\R^3)$ on $\mc F(\R^3)$ is linear, i.e. $\ph.(a I + b J) = a(\ph.I) + b (\ph.J)$ and so the tangent action on $T\mc F(\R^3) \cong \mc F(\R^3) \x \mc F(\R^3)$ coincides with the action on $\mc F(\R^3)$,
\[
\ph.(I, U) = (I \circ \ph^{-1}, U\circ\ph^{-1})\;.
\]
In particular we don't have to keep track of the basepoint. To compute the dual action on $\mc F(\R^3)^\ast$ we use the definition
\begin{align*}
\langle \ph.\pi, U \rangle_{\mc F(\R^3)^\ast \x \mc F(\R^3)}
&=\langle \pi, \ph^{-1}.U \rangle_{\mc F(\R^3)^\ast \x \mc F(\R^3)} \\
&= \int_{\R^3} \pi(x) U(\ph(x)) \ud x \\
&= \int_{\R^3} \left| \det D\ph^{-1}(x) \right| \pi\left(\ph^{-1}(x)\right) U(x) \ud x \\
&= \left\langle \left| \det D\ph^{-1}(x) \right| \pi \circ \ph^{-1}, U \right\rangle_{\mc F(\R^3)^\ast \x \mc F(\R^3)}
\end{align*}
with $\pi \in \mc F(\R^3)^\ast$ and $U \in \mc F(\R^3)$. Thus the cotangent lifted action is given by
\[
\ph.(I, \pi) = \left(I \circ \ph^{-1}, \left| \det D\ph^{-1}(x) \right| \pi \circ \ph^{-1} \right)\;,
\]
and we see that the objects dual to images transform as densities.

Now we can compute the criticality condition from Thm. \ref{lddmm_DE},
\[
DE(u)(t) = u_t^\flat + \ph_t.I_0 \diamond \ph_t\ph_1^{-1}.\pi\;,
\]
with $\pi = \frac{1}{\si^2}\left(\ph_1.I_0 - I_1\right)^\flat$. To simplify the formulas, let us define $\ph_{t,1} := \ph_t \circ \ph_1^{-1}$, which denotes the flow of $u_t$ from time 1 backwards to $t$. In general $\ph_{t,s} := \ph_t \circ \ph_s^{-1}$ is the solution of
\[
\p_t \ph_{t,s}(x) = u_t\left(\ph_{t,s}(x)\right),\qquad \ph_{s,s}(x) = x\;.
\]
So we have
\[
DE(u)(t) = Lu^t - \left| \det D\ph_{t,1}^{-1}(x) \right| \left(\pi \circ \ph_{t,1}^{-1}\right) \nabla \left(\ph_t.I_0 \right)\;,
\]
and
\begin{multline*}
\pi \circ \ph_{t,1}^{-1} = \frac 1{\si^2} \left(I_0\circ\ph_1^{-1} - I_1\right) \circ \ph_1 \circ \ph_t^{-1} = \\
= \frac 1{\si^2} \left(I_0 \circ \ph_t^{-1} - I_1 \circ \ph_1 \circ \ph_t^{-1} \right) 
= \frac 1{\si^2} \left(\ph_t.I_0 - \ph_{t,1}.I_1\right)\;.
\end{multline*}
Hence the derivative is given by
\[
DE(u)(t) = Lu_t - \frac {1}{\si^2} \left| \det D\ph_{t,1}^{-1}(x) \right| 
\left(\ph_t.I_0 - \ph_{t,1}.I_1\right) \nabla \left(\ph_t.I_0\right)\;,
\]
and critical points of $E$ satisfy
\[
Lu_t = \frac {1}{\si^2} \left| \det D\ph_{t,1}^{-1}(x) \right| 
\left(\ph_t.I_0 - \ph_{t,1}.I_1\right) \nabla \left(\ph_t.I_0\right)\;.
\]
This formula was first derived in \cite{Beg2005}, where it was used to implement a gradient descent method for $E$, which enabled computation of a numerical solution of the registration problem.

\runinhead{Conservation of Momentum.}
Returning to the general framework let us have a closer look at the equation \eqref{eq_DE} for the derivative and the information contained therein. Let $u_t$ be a critical point of the registration problem in Def. \ref{abstract_lddmm}. Then 
\begin{equation}
\label{DE_crit}
u_t^\flat = -g_t.I_0 \diamond g_tg_1^{-1}.\pi\;,
\end{equation}
which we can reformulate as
\begin{align}
u_t^\flat &= -\on{Ad}_{g_t^{-1}}^\ast \left(I_0 \diamond g_1^{-1}.\pi\right) \\
\on{Ad}_{g_t}^\ast u_t^\flat &= I_0 \diamond g_1^{-1}.\pi\;.
\label{ep_int_diamond}
\end{align}
Now note that the right hand side of \eqref{ep_int_diamond} does not depend on time any more while the left hand side doesn't depend on $V$ any more. As the right hand side is independent of $t$, we can differentiate the identity to obtain
\begin{equation}
\label{ep_int}
\p_t \left(\on{Ad}_{g_t}^\ast u_t^\flat \right) = 0\;.
\end{equation}

\runinhead{Differentiating $\on{Ad}$ and $\on{Ad}^\ast$.}
It is time to introduce some more tools from geometry related to the derivatives of the adjoint and coadjoint representations. Differentiating \eqref{ep_int} with respect to $u_t^\flat$ is not a problem, because $\on{Ad}^\ast_{g_t}$ is a linear transformation. What we need to know, is how to differentiate the expression with respect to $g_t$. 

We know from the definition of $\on{Ad}$, that it can be interpreted as a map $\on{Ad} : G \to GL(\mf g)$. The group $GL(\mf g)$ of invertible linear transformations of $\mf g$ is also a Lie group. If $\on{dim} \mf g = n$, then we can identify $GL(\mf g) \cong GL(\R^n)$ with invertible $n\x n$-matrices. Because invertible matrices form an open subset of all matrices, the tangent space $T_e GL(\R^n)$ at the identity is the space of all matrices. Thus the Lie algebra of $GL(\mf g)$ is $\mf{gl}(\mf g)$, the space of all linear transformations of $\mf g$. Hence the derivative of $\on{Ad}$ at $e \in G$ is a map
\[
\on{ad} := T_e \on{Ad} : \mf g \to \mf{gl}(\mf g), \qquad u \mapsto \on{ad}_u\;,
\]
and is called the \emph{adjoint representation} of $\mf g$. The map $\on{ad}$ figures in the following differentiation formula.
\begin{lemma}
\label{diff_Ad}
Let $t \mapsto g_t \in G$ be a smooth curve and $v \in \mf g$. Then
\[
\p_t \left(\on{Ad}_{g_t} v \right) = \on{ad}_{\p_t g_t g_t^{-1}} \on{Ad}_{g_t} v\;.
\]
\end{lemma}

\begin{proof}
We obtain this formula by writing
\begin{align*}
\p_t|_{t=t_0} \left(\on{Ad}_{g_t} v \right) 
&= \p_t|_{t=t_0} \left( \on{Ad}_{g_t g_{t_0}^{-1}} \on{Ad}_{g_{t_0}} v \right) \\
&= \on{ad}_{\p_t|_{t=t_0} g_t g_{t_0}^{-1}} \on{Ad}_{g_{t_0}} v\,.
\end{align*}
\qed
\end{proof}

However we will need the transposed version of it. For each $u \in \mf g$ fixed, the transpose $\on{ad}_u^\ast$ is defined by
\[
\left\langle \on{ad}^\ast_u \mu, v \right\rangle_{\mf g^\ast \x \mf g} = 
\left\langle \mu, \on{ad}_u v \right\rangle_{\mf g^\ast \x \mf g}\;,
\]
and thus $\on{ad}^\ast$ defines a map
\[
\on{ad}^\ast : \mf g \to \mf{gl}(\mf g^\ast)\;.
\]
This map is called the \emph{coadjoint representation} of $\mf g$. The transposed version of Lemma \ref{diff_Ad} is given in the following lemma.

\begin{lemma}
\label{diff_Adstar}
Let $t \mapsto g_t \in G$ be a smooth curve and $\mu \in \mf g^\ast$. Then
\[
\p_t \left(\on{Ad}^\ast_{g_t} \mu \right) = \on{Ad}_{g_t}^\ast \on{ad}^\ast_{\p_t g_t g_t^{-1}} \mu\;.
\]
\end{lemma}

\begin{proof}
Take $u \in \mf g$ and consider
\begin{align*}
\p_t \langle \on{Ad}^\ast_{g_t} \mu, u \rangle 
&= \langle \mu, \p_t \on{Ad}_{g_t} u \rangle \\
&= \langle \mu, \on{ad}^\ast_{\p_t g_t g_t^{-1}} \on{Ad}_{g_t} u \rangle \\
&= \langle \on{Ad}^\ast_{g_t} \on{ad}^\ast_{\p_t g_t g_t^{-1}} \mu, u \rangle\,,
\end{align*}
from which the statement follows.
\qed
\end{proof}

\runinhead{The Euler-Poincar\'e Equation.}

Lemma \ref{diff_Adstar} allows us to express equation \eqref{ep_int} as,
\begin{align*}
0 = \p_t \left(\on{Ad}_{g_t}^\ast u_t^\flat \right)
&= \on{Ad}_{g_t}^\ast \p_t u_t^\ast + \on{Ad}_{g_t}^\ast \on{ad}^\ast_{\p_t g_t g_t^{-1}} u_t^\flat \\
&= \on{Ad}_{g_t}^\ast \left( \p_t u_t^\ast + \on{ad}^\ast_{\p_t g_t g_t^{-1}} u_t^\flat \right)\;,
\end{align*}
and because $\on{Ad}_{g_t}^\ast$ is invertible we obtain the equation
\[
\p_t u_t^\flat = -\on{ad}^\ast_{u_t} u_t^\flat
\;.\]
Let us state this result as a theorem.

\begin{theorem}\label{EP-def}
Let $t \mapsto u_t \in \mf g$ be a solution of the registration problem from Def. \ref{abstract_lddmm}. Then it satisfies the equation
\begin{equation}
\label{ep_eq}
\p_t u_t^\flat = -\on{ad}^\ast_{u_t} u_t^\flat
\;.
\end{equation}
This equation is called the \emph{Euler-Poincar\'e equation} on the Lie group $G$.
\end{theorem}
\begin{remark}
The Euler-Poincar\'e equation is an evolution equation on the dual $\mf g^*$ of the Lie algebra $\mf g$, independent of $I_0, I_1$. Discussion of the history and some applications of the Euler-Poincar\'e equation can be found in \cite{Holm1998, Marsden1999}.
\end{remark}

Now let us discuss the interplay between the group of transformations and the objects that are being matched. Let $t \mapsto u_t$ be a solution of the matching problem. Then $u_t$ satisfies the Euler-Poincar\'e equation, which depends only on the geometry of the group, as encoded by $\on{ad}^\ast$, and on the chosen metric $\langle.,,\rangle_{\mf g}$ via the $\flat$-operator. The Euler-Poincar\'e equation does not see the space of objects, the action of the transformation group thereon or the particular objects $I_0, I_1$, we are trying to match. How is this possible? In order to compute $u_t$ via the Euler-Poincar\'e equation we need to supply initial conditions and these do depend $I_0, I_1$, the group action, and the inner product $\langle.,.\rangle_V$ we chose on $V$. From \eqref{DE_crit} we see that
\begin{equation}
\label{DE_crit_0}
u_0^\flat = -I_0 \diamond g_1^{-1}.\pi\;,
\end{equation}
with $\pi = \frac{1}{\si^2}(g_1.I_0 - I_1)^\flat$. So the initial value $u_0^\flat$ depends on the given objects $I_0, I_1$, on the inner product $\langle.,.\rangle_V$ via the $\flat$-map and on the group action via the momentum map.

The momentum map has yet another role to play. It allows us to reduce the dimensionality of the matching problem. Let us assume that both $G$ and $V$ are finite-dimensional. If $\on{dim} G$ is much bigger than $\on{dim} V$, then there must be a redundancy in the action of $G$ on $V$. The momentum map $\diamond : V \x V^\ast \to \mf g^\ast$ tells us that the initial condition $u_0^\flat$ will lie in the space $\on{Im}(I_0 \diamond .)$, whose dimension is at most $\on{dim} V$. Even more, we see from \eqref{DE_crit} that for each time $t$ we have $u_t^\flat \in \on{Im} \left(g_t.I_0 \diamond .\right)$. The same thing happens for infinite dimensional spaces, as we will see in the case of image matching.

\runinhead{The EPDiff Equation.} To write the Euler-Poincar\'e equation on the diffeomorphism group we first need to calculate the operators $\on{Ad}$, $\on{ad}$ and $\on{ad}^\ast$. Differentiating the conjugation $\on{conj}_\ph(\ps) = \ph\circ\ps\circ\ph^{-1}$ gives
\[
\on{Ad}_\ph u = T_{\on{Id}}(\on{conj}_\ph).u = \left(D\ph.u\right)\circ\ph^{-1}\;,
\]
with $\ph \in \on{Diff}(\R^3)$ and $u \in \mc X(\R^3)$. Now we differentiate once more, which leads to
\[
\on{ad}_u v = T_{\on{Id}}\left(\ph \mapsto \on{Ad}_{\ph} v\right).u = Du.v - Dv.u = -[u,v]\,;
\]
where $[u,v]$ is the commutator bracket of vector fields. Next we need the coadjoint action $\on{ad}^\ast$. To compute it, we take $m \in \mf X(\R^3)^\ast$ and pair it with $\on{ad}_u v$ as in \cite{Holm1998},
\begin{align*}
\langle m, \on{ad}_u v \rangle_{L^2} 
&= \int_{\R^3} m \cdot \left(Du.v - Dv.u\right) \ud x\\
&= \int_{\R^3} m^k \p_i u^k v^i - m^k \p_i v^k u^i \ud x \\
&= \int_{\R^3} m^i \p_k u^i v^k + \p_i(m^k u^i) v^k \ud x \\
&= \langle Du^T.m + Dm.u + m \on{div} u, v \rangle_{L^2}\,.
\end{align*}
We can thus write the \emph{Euler-Poincar\'e equation on the diffeomorphism group}, also called \emph{EPDiff}. It has the form
\begin{equation}
\label{eq_epdiff}
\p_t m + Dm.u + Du^T.m + \on{div}(u)m = 0\;,\qquad m = u^\flat = Lu\;.
\end{equation}
The EPDiff equation \eqref{eq_epdiff} first appeared in the context of unidirectional propagation of shallow water waves \cite{Holm1993}. In the context of planar image registration, the crests of the shallow water waves correspond to the contour lines of the image \cite{HoRaTrYo2004}.
To improve readability in \eqref{eq_epdiff}, we have omitted the subscript $t$ for the time-dependence. For the sake of completeness, we also include the coadjoint action,
\[
\on{Ad}_\ph^\ast m = \left(\det D\ph \right) D\ph^T.(m\circ\ph)\;.
\]

\runinhead{Momentum Map for Image Matching.}
The momentum map for the action of $\on{Diff}(\R^3)$ on the space $\mc F(\R^3)$ of images is $I \diamond \pi = -\pi \nabla I$. Thus \eqref{DE_crit_0} tells us that the initial momentum is of the form
\begin{equation}
\label{DE_crit_0_im}
L u_0 = \ph_1^{-1}.\pi \nabla I_0\;.
\end{equation}
As $I_0$ is fixed this means that we only have to look for the initial momenta in the subspace
\[
\on{Im} \left(I_0 \diamond .\right) = \left\{ P\nabla I_0 \,:\, P \in \mc F(\R^3) \right\}\;,
\]
elements of which are specified using only one real-valued function $P$, while the vector field $u_0$ or equivalently the momentum $Lu_0$ needs 3 functions. This reduction strategy was employed in \cite{Miller2006, Vialard2012} to solve the matching problem by estimating the initial momentum and using the EPDiff equation to reconstruct the path.

The momentum map also allows for an intuitive interpretation. Equation \eqref{DE_crit_0_im} tells us that the optimal momentum will point in the direction of the gradient of $I_0$, that is $Lu_0$ will be orthogonal to the contour lines of $I_0$. Indeed we see from
\[
L u_t = \ph_{t,1}.\pi \nabla \left(\ph_t.I_0\right)\;,
\]
that for all times the momentum is orthogonal to the contour lines of the image $\ph_t.I_0$ at time $t$. A vector field that is parallel to the contour lines will leave the image constant and since we are interested in deforming the images with the least amount of energy it is natural that the momentum wants to be orthogonal the contour lines.

\runinhead{Evolution Equations on $T^\ast V$.}
We have seen that the solution $u_t$ of the matching problem from Def. \ref{abstract_lddmm} can be expressed via the momentum map
\[
u_t^\flat = -g_t I_0 \diamond g_{t,1}.\pi\;,
\]
that satisfies the Euler-Poincar\'e evolution equation on $\mf g^\ast$:
\begin{equation}
\label{ep_eq_2}
\p_t u_t^\flat = -\on{ad}^\ast_{u_t} u_t^\flat\;.
\end{equation}
The momentum map representation can now be used to reduce the dimensionality of the problem by writing the evolution equation \eqref{ep_eq_2} directly on $T^\ast V$. Let us define the variables
\[
I_t := g_t.I_0,\qquad P_t := g_{t,1}.\pi\;.
\]
Geometrically we have $I_t \in V$ and $P_t \in T^\ast_{I_t} V \cong V^\ast$ so that the pair $(I_t, P_t)$ describes an element of $T^\ast V$. Computing the time-derivative of $I_t$ gives
\[
\p_t I_t = \p_t \left( g_t.I_0 \right) = \left( \p_t g_t \right).I_0
= u_t g_t.I_0 = u_t.I_t = \ze_{u_t}(I_t)\;.
\]
To simplify the derivation of the evolution equation for $P_t$ we will assume that the action of $G$ on $V$ is linear, as in the case of image matching. In that case the lifted actions of $G$ on $TV$ and $T^\ast V$ do not depend on the basepoint. Take $U \in V \cong T_{I_t} V$ and consider
\begin{align*}
\p_t \left\langle P_t, U \right\rangle_{V^\ast\x V} 
&= \p_t \left\langle g_{t,1}.\pi, U \right\rangle_{V^\ast\x V} \\
&= \p_t \left\langle g_1^{-1}.\pi, g_t^{-1}.U \right\rangle_{V^\ast\x V} \\
&= \left\langle g_1^{-1}.\pi, -g_t^{-1} \left(\p_t g_t \right) g_t^{-1} .U \right\rangle_{V^\ast\x V} \\
&= -\left\langle P_t, u_t.U \right\rangle_{V^\ast\x V} \\
&= \left\langle -u_t^T.P_t, U \right\rangle_{V^\ast\x V}\;.
\end{align*}
The geometrically correct expression, which holds for a general $G$-action, not just a linear one, is
\[
\p_t P_t = -T_{I_t}^\ast \ze_{u_t}.P_t\;.
\]
In case of a linear action the fundamental vector field $\ze_{u_t}$ is linear and thus we can omit the derivative and write simply $u_t^T$ for the transpose map $T_{I_t}^\ast \ze_{u_t}$ in the last line of the calculation above. Thus we obtain the following system of evolution equations on $T^\ast V$,
\begin{equation}
\label{ev_eq}
\begin{split}
\p_t I_t &= \ze_{u_t}(I_t) \\
\p_t P_t &= -T_{I_t}^\ast \ze_{u_t}.P_t \\
u_t^\flat &= I_t \diamond P_t\;.
\end{split}
\end{equation}
Note that, while we cannot completely avoid computing the vector field $u_t$, it only needs to be updated at each time step using the the variables $(I_t, P_t)$ on $T^\ast V$. 

\runinhead{Evolution Equations for Image Matching.}
Let us write out the evolution equations in the case of image matching. The action is linear and the fundamental vector fields are given by
\[
\ze_{u}(I) = -\nabla I \cdot u\;.
\]
Now we compute the transpose
\begin{align*}
\left\langle P, \ze_u(I) \right\rangle_{\mc F(\R^3)^\ast \x \mc F(\R^3)}
&= -\int_{\R^3} P(x) \nabla I(x) \cdot u(x) \ud x \\
&= \int_{\R^3} \on{div}(P u)(x) I(x) \ud x \\
&= \left\langle \on{div}(Pu), I \right\rangle_{\mc F(\R^3)^\ast \x \mc F(\R^3)}\;.
\end{align*}
Thus the evolution equations have the form
\begin{align}
\begin{split}
\p_t I_t + \nabla I_t \cdot u_t &= 0 \\
\p_t P_t + \on{div}(P_t u_t) &= 0 \\
Lu_t &= -P_t \nabla I_t\;.
\end{split}
\label{evol-eqns}
\end{align}
See also \cite{Younes2009} for a direct derivation and \cite{HoMa2004} for an explanation and classification of the cotangent lift momentum maps associated with EPDiff.

\runinhead{Matching via Initial Momentum.}
The evolution equations in \eqref{evol-eqns} allow for a reformulation of the matching problem from Def.  \ref{abstract_lddmm}. Instead of searching for paths $t \mapsto u_t \in \mf g$, we see that
any solution of the registration problem is completely determined by the initial momentum $P_0 = g_1^{-1}.\pi$. Thus we can formulate the following equivalent matching problem.

\begin{definition}[Registration Problem via Initial Momentum]
\label{abstract_initial}
Given $I_0, I_{\mathrm T} \in V$ find $P_0 \in V^\ast \cong T_{I_0}^\ast V$ which minimizes
\[
E(P_0) = \frac 12 \left| I_0 \diamond P_0 \right|_{\mf g}^2 + \frac 1{2\si^2} 
\left| I_1 - I_{\mathrm T} \right|_V^2\;,
\]
where $I_1$ is defined as the solution of
\begin{align*}
\p_t I_t &= \ze_{u_t}(I_t) \\
\p_t P_t &= -T_{I_t}^\ast \ze_{u_t}.P_t \\
u_t^\flat &= I_t \diamond P_t\;.
\end{align*}
\end{definition}

\begin{remark}
\label{rem_const_norm}
We replaced in the Def. \ref{abstract_initial} of the registration problem the integral $\int_0^1 \left| u_t \right|_{\mf g}^2 \ud t$ over the whole time interval by $\left| u_0  \right|_{\mf g}^2 = \left| I_0 \diamond P_0 \right|_{\mf g}^2$. This is justified, because if $t \mapsto u_t \in \mf g$ is a solution of the registration problem from Def. \ref{abstract_lddmm}, then its norm $\left| u_t \right|_{\mf g}^2$ is constant in time. It is possible to prove this result directly, by using the evolution equations for $(I_t, P_t)$ as follows,
\begin{align*}
\p_t|_{t=t_0} \left( \frac 12 \left| u_t \right|_{\mf g}^2\right)
&= \left\langle \p_t|_{t=t_0} \left(I_t \diamond P_t\right), u_{t_0} \right\rangle_{\mf g^\ast \x \mf g} \\
&= \p_t|_{t=t_0} \left\langle  P_t, \ze_{u_{t_0}}(I_t) \right\rangle_{\mf g^\ast \x \mf g} \\
&= -\left\langle  T^\ast_{I_{t_0}} \ze_{u_{t_0}}.P_{t_0}, \ze_{u_{t_0}}(I_{t_0}) \right\rangle_{\mf g^\ast \x \mf g}
+ \left\langle  P_{t_0}, T_{I_{t_0}} \ze_{u_{t_0}}. \ze_{u_{t_0}}(I_{t_0}) \right\rangle_{\mf g^\ast \x \mf g} \\
&= 0\;.
\end{align*}
In order to find minima for the registration problem from Def. \ref{abstract_initial}, we would need to compute the derivative of the energy $E(P_0)$ with respect to $P_0$, which would require us to differentiate the solution $I_1$ with respect to $P_0$. This can be done using a technique called \emph{adjoint equations} and is slightly more involved than the computation of the derivative in Thm. \eqref{abs_en}. Further details as well as a discussion of the numerical discretization can be found in \cite{Vialard2012}.
\end{remark}

\runinhead{Interpretation via Riemannian Geometry}
Many of the derivations, theorems and properties discussed in this section are familiar from Riemannian geometry. Let us start with the Euler-Poincar\'e equation and discuss why it arises. The registration problem in Def. \ref{abstract_lddmm} asks us to find curves $t \mapsto u_t \in \mf g$, that are minima of
\begin{equation}
\label{rm_E1}
E(u) = \frac 12 \int_0^1 \left| u_t \right|_{\mf g}^2 \ud t + \frac{1}{2\si^2} \left| g_1.I_0 - I_1\right|^2_V\;.
\end{equation}
How does Riemannian geometry arise here? A Riemannian metric $\ga$ on a manifold is an inner product on each tangent space that varies smoothly with the basepoint. On the group $G$ we have an inner product $\langle .,.\rangle_{\mf g}$ on $\mf g = T_e G$ and we can use right-multiplication to define the following Riemannian metric on the whole group,
\begin{equation}
\label{extend_met}
\ga_g(X_g, Y_g) := \left\langle X_g g^{-1}, Y_g g^{-1} \right\rangle_{\mf g},\qquad X_g, Y_g \in T_g G\;.
\end{equation}
Let $t\mapsto u_t \in \mf g$ be a curve and $t \mapsto g_t \in G$ be its flow, i.e. $\p_t g_t = u_t g_t$, $g_0 = e$. Then \eqref{rm_E1} is equivalent to
\[
E(g) = \frac 12 \int_0^1 \ga_{g_t}\left(\p_t g_t, \p_t g_t\right) \ud t + \frac{1}{2\si^2} \left| g_1.I_0 - I_1\right|^2_V\;,
\]
where we look for the minimum over all curves $t \mapsto g_t \in G$ with $g_0 =e$. Let $t \mapsto \wt g_t$ be a minimum. Then this curve also must be a minimum of
\[
E_{\mathrm{KE}}(g) = \frac 12 \int_0^1 \ga_{g_t}\left(\p_t g_t, \p_t g_t\right) \ud t\;,
\]
over the set $\left\{t \mapsto g_t \,:\, g_0=e,\; g_1=\wt g_1\right\}$ of all curves with fixed endpoints. This is exactly the definition of a geodesic in Riemannian geometry. That is, the Euler-Poincar\'e equation in the general form
\[
\p_t u_t^\flat = -\on{ad}^\ast_{u_t} u_t^\flat\;,
\]
is the geodesic equation for right-invariant metrics on Lie groups. The property used in Rem. \ref{rem_const_norm}, that the norm  $t \mapsto \left|u_t \right|_{\mf g}^2$ is constant is also a general result for geodesics in Riemannian geometry. It can be shown using the Euler-Poincar\'e equation in the following way,
\begin{align*}
\p_t \left( \frac 12 \left| u_t \right|_{\mf g}^2\right)
&= \left\langle \p_t u_t^\flat, u_t \right\rangle_{\mf g^\ast \x \mf g} 
= \left\langle -\on{ad}^\ast_{u_t} u_t^\flat, u_t \right\rangle_{\mf g^\ast \x \mf g}
= -\left\langle  u_t^\flat, \on{ad}_{u_t} u_t \right\rangle_{\mf g^\ast \x \mf g} \;.
\end{align*}
Now we use the property that $\on{ad}_u v$ is antisymmetric, i.e. $\on{ad}_u v = -\on{ad}_v u$, which implies $\on{ad}_{u_t}u_t = 0$, and conclude that $\left| u_t \right|_{\mf g}^2$ is constant in time.

\runinhead{Riemannian Geometry on $V$.}

Let us consider the left action $\ell: G \x V \to V$ of $G$ on $V$. Assume for now that the action is \emph{transitive}, i.e. for any two $I, J \in V$ there exists $g \in G$ such that $g.I = J$. Equivalently this means that for any $I \in V$ the map $\ell^I : G \to V$ is onto. If the action is not onto, we can restrict ourselves to an orbit $G.I = \left\{ g.I \,:\, g \in G \right\}$ and proceed as below.

We have an inner product $\langle .,.\rangle_{\mf g}$ on the Lie algebra, which we can extend to a right-invariant Riemannian metric $\ga^G$ on the whole group $G$ via \eqref{extend_met}. We want to project this metric to a Riemannian metric $\ga^V$ on $V$. Fix $I_0 \in V$ and let $J \in V$ be any element. Then we can write $J = g.I_0$ for some $g \in G$, not necessarily unique, due to the transitivity of the action. If $U \in T_JV$ is a tangent vector, we can write it in the form $U = X_g.I_0 = T_g\ell^{I_0}.X_g$ with some $X_g \in T_g G$ and again $X_g$ is not necessarily unique.

\begin{theorem}
The expression
\begin{equation}
\label{met_V}
\ga^V_J(U, U) = \inf_{U = X_g.I_0} \ga^G_g(X_g, X_g)\;,
\end{equation}
defines a well-defined Riemannian metric on $V$ that is independent of the choice of $I_0$.
\end{theorem}

\begin{proof}
Two things need to be proven. First, the expression on the right side must not depend on $g$ and second we have to show that $\ga^V$ is independent of $I_0$. As a first step we note that any $X_g \in T_g G$ is of the form $Xg$ with $X \in \mf g$ and thus we can rewrite the condition in the infimum of \eqref{met_V} as
\[
U = X_g.I_0 = Xg.I_0 = \ze_X(J)\;,
\]
as well as
\[
\ga^G_g(X_g, X_g) = \ga^G_g(Xg, Xg) = \left\langle X, X \right\rangle_{\mf g}\;,
\]
and hence
\[
\ga^V_J(U,U) = \inf_{U = X_g.I_0} \ga^G_g(X_g, X_g) = 
\inf_{U = \ze_X(J)} \left\langle X, X \right\rangle_{\mf g}\;.
\]
This shows that the metric $\ga^V$ is independent of both the group element $g$ used to represent $J$ as well as the choice of $I_0$ and thus everything is proven. \qed
\end{proof}

Associated to the map $\ell^{I_0} : G \to V$ is a splitting of the Lie algebra $\mf g$ into two orthogonal subspaces. Denote by $\on{Ver}(g) = \left(\on{ker} T_g \ell^{I_0}\right)g^{-1} \subseteq \mf g$ the vertical subspace. In fact $\on{Ver}(g)$ depends only on the element $J = g.I_0$ and can be described by
\[
\on{Ver}(J) = \left\{ X \in \mf g \,:\, \ze_X(J)=0 \right\}\;.
\]
The orthogonal complement of $\on{Ver}(J)$ with respect to the inner product $\langle.,.\rangle_{\mf g}$ is called the \emph{horizontal subspace},
\[
\on{Hor}(J) = \on{Ver}(J)^\perp\;.
\]

For each $J \in V$ the momentum map gives an identification between $T_I V$ and $\on{Hor}(J)$ via
\[
T_I U \ni U \mapsto (I \diamond U^\flat)^\sharp \in \on{Hor}(J)\;,
\]
where $\sharp:\mf g^\ast \to \mf g$ denotes the inverse of the $\flat$-map. To see that $(J \diamond U^\flat)^\sharp \in \on{Hor}(J)$ take any $X \in \on{Ver}(J)$ and look at
\[
\langle J \diamond U^\flat, X \rangle_{\mf g^\ast \x \mf g} = \langle U^\flat, \ze_X(J) \rangle_{V^\ast \x V} = 0\;.
\]
Surjectivity follows in finite dimensions via dimension counting and is a more delicate matter in infinite dimensions. The momentum map has the following property: for each $U \in T_J V$ the element $(J \diamond U^\flat)^\sharp \in \mf g$ realizes the infimum in \eqref{met_V}; i.e.,
\[
\ga^J(U,U) = \langle (J \diamond U^\flat)^\sharp, (J \diamond U^\flat)^\sharp \rangle_{\mf g}\;.
\]

The Riemannian interpretation of the matching problem may now be given, as follows: A solution $t \mapsto u_t$ or $t \mapsto g_t$ of the registration problem is a solution of the Euler-Poincar\'e equation \eqref{ep_eq} and thus a geodesic on the group $G$ with respect to the metric $\ga^G$. Furthermore the velocity at all times satisfies $u_t \in \on{Hor}(g_t.I_0)$. Such geodesics are called \emph{horizontal geodesics}. It follows from Riemannian geometry that the projected curve $I_t = g_t.I_0$ is a geodesic with respect to the Riemannian metric $\ga^V$. The set of evolution equations \eqref{ev_eq} are the geodesic equations on $V$ with respect to the metric $\ga^V$, written in the Hamiltonian form \cite{Vialard2012}.

Let us come back to \eqref{imreg_dist} from the introduction, which described registration as the minimization of
\[
E(g) = d_1(e, g)^2 + \tfrac 1{\si^2} d_2(g.I_0, I_1)^2\;,
\]
where $d_1(.,.)$ is a distance function on $G$ and $d_2(.,.)$ a distance function on $V$. The LDDMM framework chose $d_1(.,.)$ to be the geodesic distance with respect to the metric $\ga^G$. What the above discussion shows is that we can replace it with $d^V(.,.)$, the geodesic distance with respect to $\ga^V$; i.e. we can minimize
\[
E(J) = d^V(I_0, J)^2 + \tfrac 1{\si^2} d_2(J, I_1)^2\;,
\]
with $d_2(.,.)$ being some other metric on $V$.

For further details on the background from Riemannian geometry and the theory of group actions consult \cite{Michor2008b}. The Hamiltonian approach to Riemannian geometry, including the infinite dimensional case is described in \cite{Michor2007}. Riemannian metrics induced by group actions, especially the diffeomorphism group, in the context of shape matching are discussed in \cite{Micheli2013} and \cite{Bauer2013b_preprint}.

\section{Existence of Solutions for Image Registration}
\label{sec_analysis}

In this section we want to present a framework that allows us to prove the existence of minimizers for the image registration problem, that is for the energy
\[
E(u) = \frac 12 \int_0^1 |u_t|^2_L \ud t + \frac 1{2\si^2} \| I_0 \circ \ph^{-1} - I_1 \|^2_{L^2}\;,
\]
where $I_0, I_1 : \R^3 \to \R$ are grey-value images, $u : [0,1] \to \mf X(\R^3)$ a time-dependent vector field and $\ph_1$ its flow at time 1.

There are two competing tendencies in the mathematical modelling for image registration. We want the diffeomorphism group to be an (infinite-dimensional) Lie group. That is, we want the group operations to be smooth, so that we can rigorously apply the geometric framework of Sect. \ref{sec_geometry}. In addition, we want the Lie algebra of the diffeomorphism group with the norm $\langle.,.\rangle_L$ to be a Hilbert space, so that we can use completeness to show the existence of minimizers. Unfortunately the following theorem by Omori \cite{Omori1978} shows that these two requirements are incompatible. 

\begin{theorem}[Omori, 1978]
If a connected Banach-Lie group $G$ acts effectively, transitively and smoothly on a compact manifold, then $G$ must be a finite dimensional Lie group.
\end{theorem}

The action of a Lie group $G$ on a manifold $M$ is called effective, if
\begin{center}
$g.x=h.x$ for all $x \in M$ implies $g=h$.
\end{center}
This means that we can distinguish group elements based on how they act on the manifold. The action of the diffeomorphism group $\on{Diff}(M)$ on the base manifold $M$, given by $\ph.x = \ph(x)$ is by definition effective. The theorem thus implies that the diffeomorphism group of a compact manifold cannot be made into a Banach-Lie group. For noncompact manifolds the argument is a bit more complicated, but also follows from results in \cite{Omori1978}.

Since we cannot have both smooth group operations and a Hilbert space as a Lie algebra, we will now describe a framework that gives up the structure of a Lie group to gain completeness. For more detailed exposition and full proofs, we refer to \cite{Younes2010}.

Since none of the arguments in this section are specific to three dimensions, we will consider the case of $d$-dimensional images. Also images are not necessarily defined on the whole of $\R^d$. So let $\Om \subseteq \R^d$ be an open subset of $\R^d$, where the image $I : \Om \to \R$ is defined. We consider a certain class of spaces of vector fields, called \emph{admissible vector spaces}, to serve as the equivalent of a Lie algebra. The following introduction is taken from \cite{Bruveris2012}. 

\begin{definition}
A Hilbert space $\mathcal{H}$, consisting of vector fields on the domain $\Omega$, is called \emph{admissible}, if it is continuously embedded in $C^1_0(\Omega, \R^d)$, i.e. there exists a constant $C > 0$ such that
\[
\lvert u \rvert_{1,\infty} \leq C \lvert u \rvert_\H\;.
\]
\end{definition}

Here $C^1_0(\Omega, \R^d)$ is the space of all $C^1$-vector fields on $\Omega$ that vanish on the boundary $\partial \Omega$ and at infinity with the norm
\[
\lvert u \rvert_{1,\infty} := \sup_{x \in \Omega}~{\lvert u(x) \rvert + \sum_{i=1}^d~{\lvert \nabla u^i(x) \rvert}}\;.
\]
An admissible vector space $\mathcal{H}$ falls into the class of reproducing kernel Hilbert spaces.
\begin{definition}
A Hilbert space $\mathcal{H}$, consisting of functions $u: \Omega \to \R^d$ is called a \emph{reproducing kernel Hilbert space (RKHS)}, if for all $x \in \Omega$ and $a \in \R^d$ the directional point-evaluation $\operatorname{ev}_x^a: \mathcal{H} \to \R$ defined as $\operatorname{ev}_x^a(u) := a \cdot u(x)$ is a continuous linear functional.

In this case the relation
\[
\langle u, K(.,x)a \rangle = a \cdot u(x), \qquad u \in \mathcal{H}, \, a \in \R^d\;,
\]
defines a function $K:\Omega \times \Omega \to \R^{d\times d}$, called the \emph{kernel of $\mathcal{H}$}.
\end{definition}

If we denote by $L:\mc H \to \mc H^\ast$ the canonical isomorphism between a Hilbert space and its dual, then we have the relation
\[
K(y,x)a = L^{-1}(\on{ev}_x^a)(y)\;.
\]
In order for the RHKS to be admissible, the kernel $K$ has to satisfy the following properties:
\begin{itemize}
\item
$K$ is twice continuously differentiable with bounded derivatives, i.e. $K \in C^2(\Om\x\Om, \R^{d\x d})$ and $|K|_{2,\infty} < \infty$.
\item
$K$ vanishes on the boundary of $\Om\x\Om$, i.e. $K(x,y) = 0$ whenever $x \in \p \Om$ or $y \in \p \Om$.
\end{itemize}

Further exposition of the theory of RKHS can be found, e.g. in \cite{Aronszajn1950}, \cite{Saitoh1988}.

\begin{example}
The Sobolev embedding theorem (see e.g. \cite[Chapter 6]{Adams1975}) states that for $\Omega \subseteq \R^d$ there is an embedding
\[
H^{k+m}(\Omega) \hookrightarrow C^k(\Omega), \qquad m > \frac d 2\;,
\]
of the Sobolev space $H^{k+m}(\Omega)$ into the space of $k$-times continuously differentiable functions $C^k(\Omega)$. Therefore for $m>\tfrac d2 +1$, the space $H^{m}(\Omega)$ is an admissible space. The corresponding kernel is the Green's function of the operator $L = \on{Id} + \sum_{j=1}^{m}~{(-1)^j \Delta^j}$.
\end{example}

We fix an admissible vector space $\mathcal{H}$ with kernel $K$ and let $u \in L^2([0,1], \mathcal{H})$ be a time-dependent vector field. In Sect. \ref{sec_geometry} we assumed the vector fields to be smooth in time, but since we want to minimize over the space of time-dependent vector fields, we work here with the space $L^2([0,1], \mc H)$ of vector fields that are only square-integrable in time. This space is a Hilbert space with the inner product given by 
\[
\langle u, v \rangle_{L^2_{\mc H}} = \int_0^1 \left\langle u_t, v_t \right\rangle_{\mc H}^2 \ud t\;.
\]
We want to define the flow $\ph_t$ of the vector field $u$, as before via the differential equation
\begin{equation}
\label{flow_eq}
\partial_t \ph_t = u_t \circ \ph_t, \qquad \ph_0(x)=x\;.
\end{equation}
If $u_t$ were smooth or at least continuous in time, we could apply standard existence theorems for ODEs. Note that the theorem of Picard-Lindel\"of requires vector fields that are continuous in time and Lipschitz continuous in space. In our case $u_t$ is continuously differentiable in space, but only square-integrable in time. We have the following result concerning the existence and uniqueness of a flow for such a vector field.
\begin{theorem}
Let $\mc H$ be an admissible space and $u \in L^2([0,1], \mc H)$ a time-dependent vector field. Then \eqref{flow_eq} has a unique solution $\ph \in C^1([0,1] \times \Omega, \Omega)$, such that for each $t \in [0,1]$, the map $\ph_t:\Om \to \Om$ is a $C^1$-diffeomorphism of $\Om$.
\end{theorem}

\begin{proof} See \cite[Appendix C.2]{Younes2010} for the existence of a solution and \cite[Thm. 8.7]{Younes2010} for properties of $\ph_t$. \qed
\end{proof}

For matching purposes we will work with all diffeomorphisms that can be obtained as flows of such vector fields. Define the group $G_{\mc H}$ to be
\begin{equation}
\label{def_diff}
G_{\mc H} := \left\{ \ph_1 : \ph_t \textrm{ is a solution of \eqref{flow_eq} for some } u \in L^2([0,1], \mathcal{H}) \right\}\;.
\end{equation}
It can be equipped with the following distance, which is modelled after the geodesic distance on Riemannian manifolds,
\begin{equation}
\label{def_diff_dist}
d_{\mc H}(\ps_0, \ps_1)^2 = \inf_{u \in L^2([0,1], \mc H)} \left\{ \int_0^1 \left|u_t\right|_{\mc H}^2 \ud t \,:\, \ps_1 = \ps_0 \circ \ph_1^u \right\}\;.
\end{equation}
The set $G_{\mc H}$ has the following properties

\begin{theorem}
Let $\mc H$ be an admissible space and $G_{\mc H}$ defined via \eqref{def_diff}. Then
\begin{itemize}
\item
$G_{\mc H}$ is a group.
\item
(Trouv\'e) The function $d_{\mc H}$ is a distance on $G_{\mc H}$ and $(G_{\mc H}, d_{\mc H})$ is a complete metric space.
\item
For each $\ps_0, \ps_1 \in G_{\mc H}$ there exists $u \in L^2([0,1], \mc H)$ realizing the infimum in \eqref{def_diff_dist}, i.e. $d_{\mc H}(\ps_0, \ps_1) = |u|_{L^2_{\mc H}}$.
\end{itemize}
\end{theorem}

\begin{proof}
See \cite[Thm. 8.14]{Younes2010} for a proof that $G_{\mc H}$ is closed under group operations, see \cite[Thm. 8.15]{Younes2010} for the completeness of $d_{\mc H}$ and see \cite[Thm. 8.20]{Younes2010} for the existence of a minimum. \qed
\end{proof}

Note that we have not said anything about the structure of $G_{\mc H}$ as a manifold or a Lie group. In an informal way the space $\mc H$ acts as a ``Lie algebra'' of the ``Lie group'' $\mc G_{\mc H}$, but all the statements of Sect. \ref{sec_geometry} are to be interpreted only formally in this framework.

The main advantage of working with admissible spaces and the group $G_{\mc H}$ is the following theorem.

\begin{theorem}
\label{thm_ex_min}
Let $\mc H$ be an admissible space and $I_0, I_1 \in L^2(\Om)$. Then there exists a minimizer for the registration energy
\begin{equation}
\label{reg_en}
E(u) = \frac 12 \int_0^1 |u_t|^2_{\mc H} \ud t + \frac 1{2\si^2} \| I_0 \circ \ph^{-1} - I_1 \|^2_{L^2}\;,
\end{equation}
i.e. there exists $\wt u \in L^2([0,1], \mc H)$ such that
$E(\wt u) = \inf_{u \in L^2([0,1], \mc H)} E(u)$.
\end{theorem}

\begin{proof}[Sketch]
Let us introduce the notation $U(\ph) = \frac 1{2\si^2} \| I_0 \circ \ph^{-1} - I_1 \|^2_{L^2}$. This allows us to write
$ E(u) = \frac 12 |u|_{L^2_{\mc H}}^2 + U(\ph_1)$. Consider a minimizing sequence $u^n \in L^2([0,1], \mc H)$, such that $E(u^n) \to \inf_u E(u)$. As the functional $U(.)$ is bounded from below, the sequence $(u^n)_{n \in \mathbb N}$ is bounded in the Hilbert space $L^2([0,1], \mc H)$. Since bounded sets in Hilbert spaces are weakly compact, we can extract a subsequence, again denoted by $(u^n)_{n \in \mathbb N}$, that converges weakly to some $\wt u$. What remains to show now is that this $\wt u$ is indeed the minimizer. The inequality $\inf_u E(u) \leq E(\wt u)$ is trivial and it remains to show the converse.

From
\[
\langle u^n, \wt u \rangle \leq \lvert u^n\rvert_{L^2} \lvert \wt u \rvert_{L^2}
\]
we see by passing to the $\liminf$ that $\lvert \wt u \rvert_{L^2} \leq \liminf_{n\to \infty} \lvert u^n\rvert_{L^2}$. Concerning $U(\ph^n_1)$ we will use the following property:
\begin{center}
  If $u^n \to \wt u$ weakly in $L^2([0,1], \mc H)$, then $U(\ph^n_1) \to U(\wt \ph_1)$.
\end{center}
This implication can be split up into two steps.
\begin{enumerate}
\item
Let $u^n \to \wt u$ weakly. Then the sequence $(\ph^n_1)_{n \in \mathbb N}$ of flows satisfies
\begin{itemize}
\item $\ph^n_1 \to \wt \ph_1$ and $(\ph^n_1)^{-1} \to \wt \ph_1^{-1}$ uniformly on compact sets and
\item the sequence $(|D\ph^n|_\infty)_{n \in \mathbb N}$ is bounded.
\end{itemize}
\item
Under the above conditions on the sequence $(\wt \ph_1^n)_{n \in \mathbb N}$ of flows, we have convergence $U(\ph_1^n) \to U(\wt \ph_1)$.
\end{enumerate}
The proof for the first step is a combination of \cite[Thm. 8.11]{Younes2010} and Gronwall's lemma \cite[Thm. C.8]{Younes2010}. An explicit proof for the second step can be found in \cite[Thm. 2.7]{Bruveris2012}. Putting all pieces together we get
\begin{align*}
E(\wt u) &= \tfrac 12 \lvert \wt u\rvert^2_{L^2_{\mc H}} + U(\wt \ph_1) \\
& \leq \liminf_{n \to \infty} \tfrac 12\lvert u^n \rvert^2_{L^2_{\mc H}}  + \lim_{n \to \infty} U(\ph^n_1)
= \lim_{n \to \infty} E(u^n)\\
& \leq \inf_{u \in L^2} E(u)\enspace,
\end{align*}
Hence $\wt u$ is a minimizer. \qed
\end{proof}

To make the connection back to the general framework, we will show that, if the images $I_0, I_1$ are sufficiently smooth, then the minimizer will also be smooth, both in space and in time. Thus the smooth geometric framework on Sect. \ref{sec_geometry} retains some use. It may not be sufficient to show existence of a minimizer or its properties, but if existence has been established, the minimizer does reside in the smooth framework.

\begin{theorem}
Let $\mc H$ be an admissible space. Let $I_0 \in C_0^1(\Om)$ and $I_1 \in C_0(\Om)$. Then the minimizer $u$ of the registration energy \eqref{reg_en} satisfies
\[
Lu_t = \frac {1}{\si^2} \left| \det D\ph_{t,1}^{-1}(x) \right| 
\left(I_0\circ\ph_t^{-1} - I_1\circ\ph_{t,1}^{-1}\right) \nabla \left(I_0\circ\ph_t^{-1}\right)\;,
\]
and the equation
\begin{equation}
\label{younes_cons_mom}
L u_t = \on{Ad}_{\ph_t^{-1}}^\ast Lu_0\;.
\end{equation}
\end{theorem} 
\begin{proof}
See \cite[Thm. 11.5]{Younes2010} and \cite[Thm. 11.6]{Younes2010}. \qed
\end{proof}

We did encounter equation \eqref{younes_cons_mom} in the smooth setting as well in the form \eqref{ep_int}. Now however we see that the right hand side is differentiable in $t$, because $\ph_t$, being the solution of a differential equation, is differentiable in $t$ and so is $u_t$. Differentiating \eqref{younes_cons_mom} with respect to $t$ leads to the EPDiff equation \eqref{eq_epdiff}. Finally we can state the following theorem, which brings us back to the smooth framework.

\begin{theorem}
Let $\mc H$ be an admissible space with $\Om = \R^3$. Then $H^\infty(\R^3, \R^3) \subset \mc H$ and $\on{Diff}_{H^\infty}(\R^3) \subset G_{\mc H}$. If $I_0, I_1 \in C^\infty_c(\R^3)$, then the minimizer $t \mapsto u_t$ of \eqref{reg_en} from Thm. \ref{thm_ex_min} satisfies
\[
u \in C^\infty([0,1], H^\infty(\R^3, \R^3))\;.
\]
\end{theorem}

This theorem closes the loop between the geometric and analytic settings for image registration.


\end{document}